\documentclass{amsart}
\usepackage{amsmath, amssymb, epsfig, graphicx, color }
\usepackage[all]{xy}
\usepackage{enumerate,comment}
\usepackage{yfonts}
\usepackage{hyperref}

\newtheorem{thm}{Theorem}[section]

\newtheorem{lem}[thm]{Lemma}
\newtheorem{prop}[thm]{Proposition}
\newtheorem{cor}[thm]{Corollary}

\newtheorem{ex}[thm]{Example}

\newtheorem{fact}[thm]{Fact}

\theoremstyle{definition}
\newtheorem{remark}[thm]{Remark}
\newtheorem{question}[thm]{Question}

\newcommand{\reals}{\mathbb R}
\newcommand{\nat}{\mathbb N}
\newcommand{\cantor}{ 2^{\nat}}

\newcommand{\el}{{\ell _1}}
\newcommand{\baire}{ \nat ^ \nat}
\newcommand{\e}{\varepsilon}
\newcommand{\diam}{\text{diam}}
\newcommand{\NN}{\mathbb N}
\newcommand{\TT}{\mathbb T}

\def\rond{\mathaccent"7017 }
\newcommand{\cont}[2]{{\mathcal C} ( {#1} ,  {#2}) }
\newcommand{\hyp}[1]{{\mathcal K}( {#1})}

\title{Some universality results for dynamical systems}
\author[Darji] {Udayan B.~Darji}
\author[matheron]{\'Etienne Matheron}
\email[Darji]{ubdarj01@louisville.edu}
\email[Matheron]{etienne.matheron@univ-artois.fr}
\address[Darji]{Department of Mathematics, University of Louisville, Louisville, KY 40292, USA.}
\address[Matheron]{ Laboratoire de Math\'ematiques de Lens, Universit\'e d'Artois, Rue Jean Souvraz S.
P. 18, 62307 Lens (France).}
\thanks{The first author would like to acknowledge the hospitality and financial support of Universit\'e d'Artois}
	
\begin{document}
\begin{abstract}
We prove some ``universality" results for topological dynamical systems. In particular, we show that for any continuous self-map $T$ of a perfect  Polish space, one can find a dense, $T$-invariant set homeomorphic to the Baire space $\baire$; that there exists a bounded linear operator $U: \el \rightarrow \el$ such that any linear operator $T$ from
a separable Banach space into itself  with $\Vert T\Vert\leq 1$ is a linear factor of $U$; and that 
 given any $\sigma$-compact family ${\mathcal F}$ of continuous self-maps of
a compact metric space, there is a continuous self-map $U_{\mathcal F}$ of $\NN^\NN$ such that each $T\in {\mathcal F}$ is
a factor of $U_{\mathcal F}$.  
\end{abstract}
\maketitle
\section{Introduction}
It is well known that the Cantor space $\cantor$, the Baire space $\NN^\NN$ and the Banach space $\ell_1=\ell_1(\NN)$ have some interesting universality properties: every compact metric space is a continuous image of $\cantor$, every Polish space is a continuous image of $\baire$, and every separable Banach space is a quotient of $\ell_1$.

In this note, we will mainly be concerned with the following kind of generalization of these classical facts: instead of just lifting the points of a given space $X$ into one of the above spaces, one would like to lift continuous self-maps of $X$. 

\smallskip
Throughout the paper, by a  \emph{map} we always mean a continuous mapping between topological spaces. For us, a \emph{dynamical system} will be a pair $(X,T)$, where $X$ is a topological space and $T$ is a self-map of $X$. 

A dynamical system is $(X,T)$ is a \emph{factor} of a dynamical system $(E,S)$, and $(E,S)$ is an \emph{extension} of $(X,T)$, if  there is a 
 map $\pi$ from $E$ {\bf onto} $X$ such
that $T\pi=\pi S$, \mbox{\it i.e.} the following diagram commutes:
\[
\xymatrix{E   \ar[rr]^{S}  \ar@{->>}[d]_{\pi}                    & & E \ar@{->>}[d]_{\pi } \\
         X  \ar[rr]^{T}  & & X }
\]

When $T$ and $S$ are linear operators acting on Banach spaces $X$ and $E$, we say that $T$ is a \emph{linear} factor of $S$ if the above diagram can be realized with a linear factoring map $\pi$.

\smallskip
Among other things, we intend to prove the following two  ``common extension" results. 

\begin{itemize}
\item[$\bullet$] \emph{There exists a bounded linear operator $U: \el \rightarrow \el$ such that every linear operator $T$ from
a separable Banach space into itself  with $\Vert T\Vert\leq 1$ is a linear factor of $U$.} (This is proved in Section \ref{linear}.)
\item[$\bullet$]  \emph{Given any $\sigma$-compact family ${\mathcal F}$ of self-maps of
a compact metric space, there is a map $U_{\mathcal F}:\baire\to\NN^\NN$ such that each $T\in {\mathcal F}$ is
a factor of $U_{\mathcal F}$.} (This is proved in Section \ref{Bairelift}.)
\end{itemize}

\smallskip
Regarding the Baire space $\baire$, we will also show that any Polish dynamical the dynamics of any continuous self-map $T$ of a perfect  Polish space $X$ is in some sense ``captured" by a self-map of $\baire$; namely, there is a dense, $T$-invariant set inside $X$ which is homeomorphic to $\baire$. This result is proved in Section~\ref{subsystem}.

\smallskip
The reader may wonder why, in the second result quoted above, the domain of the map $U_{\mathcal F}$ is  $\baire$ and not the Cantor space $\cantor$. The reason is that, even for simple compact families $\mathcal F$, one cannot obtain a common extension defined on a compact metrizable space; see Remark \ref{biensur}. We have not been able to characterize those $\sigma$-compact families ${\mathcal F}$ admitting a common extension defined on $\cantor$. However, we give a simple sufficient condition which implies in particular that this holds true for the family of all \emph{contractive} self-maps of a compact metric space. This is proved in Section \ref{nonsense}. We also show that for an arbitrary $\sigma$-compact family $\mathcal F$, a rather natural relaxation in the definition of an extension allows this; see Section \ref{generalized}.

\smallskip
The kind of question we are considering here may be formulated in a very general way. One has at hand a certain category $\mathfrak C$ and, given a collection of objects $\mathcal F\subseteq\mathfrak C$, one wants to know if there exists an object $\mathtt U_{\mathcal F}\in\mathfrak C$ which is \emph{projectively universal} for $\mathcal F$, in the following sense: for any $\mathtt T\in\mathcal F$, there is an epimorphism $\pi:\mathtt U_{\mathcal F}\to \mathtt T$ from $\mathtt U_{\mathcal F}$ to $\mathtt T$. Less ambitiously, one may just want to find a ``small" family of objects $\mathcal U$ which is projectively universal for $\mathcal F$, \mbox{\it i.e.} for any $\mathtt T\in\mathcal F$, there exists an epimorphsim from some $\mathtt U\in\mathcal U$ to $\mathtt T$.

In this paper, we have been concerned with categories $\mathfrak C$ whose objects are topological dynamical systems and whose morphisms are factoring maps. For example, if $\mathfrak C$ is the category of all compact metrizable dynamical systems then, as mentioned above, there is no projectively universal object for $\mathfrak C$, but there is one for the family of all contractive dynamical systems defined on a given compact metric space. Moreover, it is well known that every self-map of a compact metrizable space is a factor of some self-map of $\cantor$ (see Section \ref{extensionppty}); in other words, the family of all dynamical systems defined on $\cantor$ is projectively universal for the whole category. Likewise, if $\mathfrak C$ is the category all linear dynamical systems on separable Banach spaces then (as annouced above) there is an object $\mathtt U$ living on $\ell_1$ which is projectively universal for the family all $(X,T)\in\mathfrak C$ with $\Vert T\Vert\leq 1$.

\smallskip
Many other categories may be worth studying from this point of view.  We point out three of them.

One is the category of Polish dynamical systems, for which the family of all dynamical systems defined on $\NN^\NN$ happens to be projectively universal (see Section \ref{extensionppty}). 

The second category we have in mind is that of dynamical systems defined on \emph{arc-like continua}. 
An exotic space called the \emph{pseudo-arc} plays a central role in the theory of arc-like continua. As it turns out, the family of all dynamical systems defined on the pseudo-arc is projectively universal for this category (this is proved in \cite{lewis3}). The pseudo-arc has a rich and long history; see \mbox{\it e.g.}  \cite{lewis1} for a detailed survey and 
\cite{is} for an interesting model theoretic construction.

A third interesting example is the category of all \emph{minimal $G$-flows}, for a given topological group $G$. Recall that a $G$-flow $(X,G)$ is a compact topological space $X$ endowed with a continuous action of $G$, and that a $G$-flow $(X,G)$ is said to be minimal if every $G$-orbit $\{ g\cdot x;\; g\in G\}$ is dense in $X$. It is well known that there exists a projectively universal object $(M(G),G)$ for this category, which is unique up to isomorphism. There is a vast literature on universal minimal flows (see \cite{pestov2} and the references therein, for example  \cite{gw1}, \cite{gw2}, \cite{kpt}, \cite{julien}, \cite{pestov}, \cite{uspenskij}). In particular, the problem of determining exactly when $M(G)$ is \emph{metrizable} has been well investigated; see \cite{julien}. Very much in the spirit of the present paper, it seems quite natural to ask for which families $\mathcal F$ of minimal $G$-flows one can find a projectively universal $G$-flow defined on a metrizable compact space.

\section{Dense subsystems homeomorphic to $\NN^\NN$}\label{subsystem} In this section, our aim is to prove the following result.

\begin{thm}\label{thm:bairerestrict}
Let $X$ be a perfect Polish space. For any self-map $T: X \rightarrow X$, there is a dense $G_\delta$ set $Y \subseteq X$ homeomorphic to $\baire$ such that $T(Y) \subseteq Y$.
\end{thm}

\smallskip
This will follow from Lemmas \ref{lem:restrictnwlc} and \ref{lem:baireinv} below. 

\smallskip
We start with a lemma where no assumption is made on the Polish space $X$. Recall that a space $E$ is \emph{$0$-dimensional} if it has a basis consisting of clopen sets. 

\begin{lem}\label{lem:0dimext} Let $X$ be a Polish space, $T: X \rightarrow X$
be continuous and $Z \subseteq X$ be $0$-dimensional with $T(Z) \subseteq Z$. 
Then, there is $Y$
such that
\begin{enumerate}
\item[\rm (1)] $Z \subseteq Y \subseteq X$ and $T(Y) \subseteq Y$;
\item[\rm (2)] $Y$ is $0$-dimensional and Polish (\mbox{\textit{i.e.}} a $G_{\delta}$ subset of $X$). 
\end{enumerate}
\end{lem}
\begin{proof} We need the following simple and well known fact, which easily follows from basic definitions.
\begin{fact}\label{fact:fact} Let $E$ be Polish and let $F\subseteq E$ be $0$-dimensional. Then, for any $\e>0$, 
there is a sequence of pairwise disjoint open set $\{U_n\}$ in $E$ such that $\diam(U_n) < \e$ and 
 $F \subseteq \bigcup_n U_n$.
 \end{fact}

By Fact~\ref{fact:fact} we may choose a sequence $\{U^1_n\}$ of pairwise disjoint  open subsets of $X$ 
such that the diameter of each $U^1_n$ is less than 1 and $Z \subseteq \bigcup_n U^1_n$. Now let $V^1_n := T^{-1}(U^1_n)$.
Then the $\{V^1_n\}$ are pairwise disjoint open sets. Moreover, as $T(Z) \subseteq Z$, we have that $Z \subseteq 
\bigcup_{n=1}^{\infty} V^1_n$.  Now we apply Fact~\ref{fact:fact} with $E:=V^1_m$ and $F:=V^1_m \cap Z$, for each $m\in\NN$. This produces a collection of pairwise disjoint open sets $\{U^2_n\}$ in $X$ such that, for each $n$, 
the diameter of $U^2_n$ is less than $1/2$  and $U^2_n \subseteq V^1_m$ for some $m$. Hence the collection $\{U^2_n\}$ satisfy the following properties:
\begin{itemize}
\item $Z \subseteq \bigcup_{n=1}^{\infty} U^2_n$,
\item $\diam (U^2_n) < 1/2$ for each $n$,
\item $T(U^2_n) \subseteq U^1_m$ for some $m$.
\end{itemize}
Continuing in this fashion, we get for each $j\in\mathbb N$ a collection of pairwise disjoint open sets $\{U^j_n;\; n\in\NN\}$ such that
the following properties hold:
\begin{itemize}
\item $Z \subseteq \bigcup_{n=1}^{\infty} U^j_n$,
\item $\diam (U^j_n) < 1/2^j$ for each $n$,
\item $T(U^j_n) \subseteq U^{j-1}_m$ for some $m$.
\end{itemize}

Now we simply let $Y := \bigcap_{j=1}^{\infty} \bigcup_{n=1}^{\infty} U^j_n$. This $Y$ has the required properties.
\end{proof}

\begin{remark}
In the above lemma, one cannot make $Y$ dense in $X$. For example, take $X $ to be the disjoint union of $\reals$ and a countable discrete set $\{ a_n;\; n\in\NN\}$. The topology is the usual one. Let the map $T$ on $X$ defined as follows: it is the identity on $\reals$ and it maps $a_n$
to the $n^{th}$ rational in $\reals$ (under some fixed enumeration). Now take $Z$ to be
the set of irrationals on the line. There is no way to extend this $Z$ to a dense, $T$-invariant and still
 $0$-dimensional set $Y$. 
\end{remark}
\begin{lem}\label{lem:restrictnwlc}
Let $X$ be a perfect Polish space and $T:X \rightarrow X$ be continuous. Then, there is a Polish $X_0 \subseteq X$ dense in $X$ and nowhere locally compact such that $T(X_0) \subseteq X_0$. 
\end{lem}
\begin{proof} We will find a meager, $F_{\sigma}$ set $Z \subseteq X$ dense in $X$ such that $T^{-1}(Z) \subseteq Z$. It will then suffice to let $X_0:= X \setminus Z$. As $X$ is Polish and $Z$ an $F_{\sigma}$ meager set, we have that $X_0$ is Polish and dense in $X$. Moreover, $X_0$ is nowhere locally compact. Indeed, assume that for some nonempty open set $U$ in $X$, the set
$X_0 \cap U$ has compact closure in $X_0$. Then, $\overline{(X_0 \cap U)}\cap X_0$ is compact and hence is closed in $X$, and it has empty interior in $X$ since $X\setminus X_0$ is dense in $X$. In particular, $X_0\cap U$ is nowhere dense in $X$, a contradiction because $X_0$ is dense in $X$.

To complete the proof, let us construct the desired $Z$.  Let $B_1, B_2, \ldots$ be an enumeration of some countable basis of open sets for $X$. Let 
\[
S= \{(n,m) \in \nat \times \nat \;:\;  T^m(B_n) \text{ is meager}\}.
\]
Let $U = \bigcup_{(n,m)\in S} B_n$. Let $A =\bigcup \{T^m(B_n);\;(n,m) \in S\}$. Then, $A$ is meager in $X$.
Let $D$ be a countable dense subset of $X \setminus A$. Clearly, each of $T^{-m}(D)$, $m\geq 0$ is $F_{\sigma}$. 
We claim, moreover, that $T^{-m}(D)$ is meager. This is clear for $m=0$ as $D$
is countable. To obtain a contradiction, assume that for some $ m \ge 1$, $T^{-m}(D)$ is not meager.
As $T^{-m}(D)$ is $F_{\sigma}$, it must contain a nonempty open set. Let $n \in \nat$ be such that
$B_n \subseteq T^{-m}(D)$.  Then, we have that $T^m(B_n)$ is meager since it is a subset of the countable
set $D$. Hence, $T^m(B_n) \subseteq A$, contradicting that $D\cap A = \emptyset$.

Let $Z = \bigcup_{m \ge 0}T^{-m} (D)$. Then, $Z$ is an $F_{\sigma}$, dense, and meager subset of $X$.
Moreover, $T^{-1}(Z) \subseteq Z$, concluding the proof of the lemma.
\end{proof}

\begin{lem}\label{lem:baireinv} Let $X$ be a nowhere locally compact Polish space and let $T:X \rightarrow X$ be continuous. Then there is a dense $G_\delta$ set $Y\subseteq X$ such that $T(Y)\subseteq Y$ and $Y$ is homeomorphic to $\baire$. Moreover, one may require that $Y$ contains any prescribed countable set $C\subseteq X$.
\end{lem}
\begin{proof} Let $C$ be any countable subset of $X$. As $X$ is separable, we may expand $C$ to be a countable set which is dense in $X$. 
Let $Z = \bigcup_{n=0}^{\infty} T^n(C)$. Clearly, $T(Z) \subseteq Z$. Moreover, $Z$ is $0$-dimensional
as it is countable.  Let $Y$ be given by Lemma~\ref{lem:0dimext}. Only that $Y$ is homeomorphic to $\baire$ needs an argument.
We already know that $Y$ is $0$-dimensional and Polish. 
Finally, $Y$ is nowhere locally compact as $X$ is and $Y$ is dense in $X$. Hence, $Y$ is homeomorphic to
 $\baire$. 
\end{proof}
\begin{remark}\label{cor:banachbaire}
It follows in particular that if $T:G\to G$ is a self-map of a non-locally compact Polish group $G$ then, for any given countable set $C \subseteq G$, there is $Y \subseteq G$ homeomorphic to $\baire$ such that $C \subseteq Y$
and $T(Y) \subseteq Y$.
\end{remark}

\begin{proof}[Proof of Theorem \ref{thm:bairerestrict}] This is now immediate by applying Lemma~\ref{lem:restrictnwlc} and then Lemma~\ref{lem:baireinv}.
\end{proof}

\section{A universal linear operator}\label{linear} In this section, we prove the following result, which says essentially that any bounded linear operator on a separable Banach space is a linear factor of a single ``universal" operator acting on $\ell_1$.
\begin{thm}\label{thm: l1} There exists a bounded linear operator $U: \el \rightarrow \el$ with $\Vert U\Vert= 1$
such that any bounded linear operator $T:X \rightarrow X$ on a separable Banach space $X$ is 
a linear factor of $\rho\cdot U$ for every $\rho\geq \Vert T\Vert$. 
\end{thm}
\begin{proof} We need the following fact. This is probably well known but we include an outline of the proof for completeness.
\begin{fact}\label{lem:universal121}
There exists a 1-1 map $\mu :\nat \rightarrow \nat$ such that, for any other 1-1 map $\sigma : \nat \rightarrow \nat$, 
one can find a set $A \subseteq \nat$ and a bijection $\pi_A: \nat \twoheadrightarrow A$ such that $\mu(A)=A$ and $\sigma = \pi_A^{-1}\mu\pi_A$.
\end{fact}
\begin{proof}[Proof of Fact \ref{lem:universal121}] For any 1-1 map $\sigma : \nat \rightarrow \nat $, let $G_{\sigma}$ be the digraph on $\nat$ induced by $\sigma$; that is :
 $\overrightarrow{ij}$ is a directed edge of $G_{\sigma}$ iff $\sigma(i)=j$. Then, each component $C$ of $G_{\sigma}$
has one of the following forms:
\begin{itemize}
\item $C$ is a cycle of length $n$ for some $n \in \nat$,
\item there is $i_0 \in C$ such that $C= \{i_0,i_1,i_2, \ldots\}$ where all $i_k$'s are distinct and the only
directed edges are of the form $\overrightarrow{i_ki_{k+1}}$, or
\item $C = \{\ldots, i_{-1}, i_0, i_1, \ldots \}$ where all $i_k$'s are distinct and the only
directed edges are of the form $\overrightarrow{i_ki_{k+1}}$.
\end{itemize}
Now, consider any $\mu$ whose digraph contains infinitely components 
of each type described above.
\end{proof}

To prove the theorem, we first observe that it suffices to construct  $U$ so that every bounded linear 
operator with norm at most 1 (on a separable Banach space) is a linear factor of $U$. 

Let $\mu$ be as in Fact~\ref{lem:universal121}. Let $\{e_n\}$ be the standard basis of $\el$, i.e.,
$e_n$ is $0$ in every coordinate except the $n^{th}$ coordinate where it takes the value $1$. We define $U: \el \rightarrow \el$ simply 
by $U(e_n) = e_{\mu (n)}$. It is clear that $U$ is a bounded linear operator of norm 1.  Let
$T: X \rightarrow X$ be a bounded linear operator acting on a separable Banach space $X$, with $\|T\| \leq 1$. Since $X$
is separable and $\|T\|\leq 1$, one can find a countable dense subset $D$ of $B_X$,
the unit ball in $X$, so that $T(D) \subseteq D$.  This may be done by starting with any countable set
dense in the unit ball and taking its union with its forward orbit. We let $\{z_n\}$ be a enumeration of $D$
so that each element of $D$ appears infinitely often in $\{z_n\}$. Now we define
$\sigma : \nat \rightarrow \nat$. For any $i \in \nat$, we have $T(z_i) = z_j$ for some $j$. Moreover, since every element of 
$D$ appears infinitely often in the sequence $\{z_n\}$, we may choose a 1-1 map $\sigma:\NN\to\NN$ such that $T(z_i)=z_{\sigma (i)}$ for all $i\in\NN$. 
Now, let $A \subseteq \nat$
and $\pi_A$ be as in Fact~\ref{lem:universal121}. Then there is a uniquely defined bounded linear operator $\pi : \el 
\rightarrow X$ (with $\Vert\pi\Vert\leq 1$) such that
\begin{displaymath}
   \pi(e_i)  = \left\{
     \begin{array}{lr}
       z_{\pi_A^{-1}(i)} & \hbox{if $i \in A$,}\\
       0 & \hbox{if  $i \notin A$.}
     \end{array}
   \right.
\end{displaymath} 

The operator $\pi$ is \emph{onto} because $\pi(B_{\el})$ contains all of $D$, which is dense
in $B_X$; see \mbox{e.g.} \cite[Lemma 2.24]{FHetc}. 

Finally, let us show show that $T$ is a factor of $U$ with witness $\pi$, \mbox{\textit{i.e.}}. $\pi U=T\pi$. It is enough to check that $\pi U$ and
$T\pi$ are equal on $\{e_i\}$ as all maps are linear and continuous and $\{e_i\}$ 
spans $\el$.  Let $i \notin A$. Then, by Fact~\ref{lem:universal121} $\mu(i) \notin A$.
Hence, we have that neither $i$ nor $\mu(i)$ is in $A$, so that
\[ \pi(U(e_i))= \pi(e_{\mu(i)}) =0\qquad{\rm and}\qquad  T(\pi(e_i)) = T(0) =0.
\]
Now suppose $i \in A$. Then, 
\[ T(\pi(e_i)) = T(z_{\pi_A^{-1}(i)}) = z_{\sigma(\pi_A^{-1}(i))}= z_{\pi_A^{-1}(\mu(i))}=\pi(e_{\mu(i)}) = \pi (U(e_i)).
\]
\end{proof}

\begin{remark}\label{obvious} One really has to consider all multiples of $U$: it is not possible to find a single bounded linear operator $U$ 
so that every bounded linear operator $T: X \rightarrow X$ is a linear factor of it. Indeed, if an operator $T$ is a linear factor of some operator $U$, then there is a constant $C$ such that $\Vert T^n\Vert\leq C\, \Vert U^n\Vert$ for all $n\in\NN$; in particular, $\lambda Id$ cannot be a factor of $U$ if $\vert \lambda\vert>\Vert U\Vert$. To see this, let $\pi$ witness the fact that $T:X \rightarrow X$ is a factor of $U:Y \rightarrow Y$. By the Open Mapping Theorem, there a constant $k$
such that $B_X \subseteq k\pi(B_{Y}) $. Let $A \ge 1$ be such that $\|\pi \| \le A$. Then, $C = k A$ has the property
that $\Vert T^n\Vert\leq C\, \Vert U^n\Vert$ for all $n\in\NN$. 
\end{remark}

{In spite of this last remark, one can find a single \emph{Fr\'echet} space operator which is universal for separable Banach space operators. (Interestingly, it is not known if there exist quotient universal Fr\'echet spaces.)} 
\begin{cor} Let $E:=(\ell_1)^\NN$, the product of countably many copies of $\ell_1$, endowed with the product topology. There exists a continuous linear operator $U_E:E\to E$ such that any bounded operator $T$ acting on a separable Banach space is a linear factor of $U_E$.
\end{cor}
\begin{proof} Let $U:\ell_1\to\ell_1$ be the operator given by Theorem \ref{thm: l1}. Then define $U_E:E\to E$ by 
$$U_E(x_1,x_2,x_3,\dots ):=(U(x_1), 2\cdot U(x_2), 3\cdot U(x_3),\dots )\, .$$
This is indeed a continuous linear operator, and it is readily checked that $U_E$ has the required property. Indeed, if $T:X\to X$ is a bounded linear operator ($X$ a separable Banach space) then one can find $n\in\mathbb N$ such that $T$ is a linear factor of $n\cdot U$, with factoring map $\pi:\ell_1\to X$. Then the operator $\widetilde \pi:E\to X$ defined by $\widetilde \pi (x_1,x_2,\dots ):= \pi(x_n)$ shows that $T$ is a linear factor of $U_E$.
\end{proof}

\section{Common extensions of compact dynamical systems}

In this section, we prove some ``common extension" results for self-maps of \emph{compact metric spaces}. In what follows, we denote by $\mathfrak D$ the collection of all dynamical systems $(X,T)$ where $X$ is a compact metric space; or, equivalently, the collection of all self-maps of compact metric spaces.

\subsection{The extension property}\label{extensionppty} A map $\pi:E \rightarrow X$ between topological spaces has the {\em extension property} if, for every map $\phi:E \rightarrow X$, there is a map $S=S_{\phi}:E \rightarrow E$ such that $\phi =\pi S$. Note that such a map $\pi$ is necessarily onto (consider constant maps $\phi$). The following theorem can be found in \cite{W}.

\begin{thm}\label{Weihr} For any compact metric space $X$, there exists a map $\pi:\cantor\to X$ with the extension property.
\end{thm}

This is a strengthening of the very well known result saying that every compact metric space is a continuous image of $\cantor$. In fact, from this lemma one gets immediately the following result from \cite{BS}. (It seems that the two results were proved independently at about the same time.)
\begin{cor}\label{BalcSim} Every $(X,T)\in\mathfrak D$ is a factor of some dynamical system of the form $(\cantor ,S)$.
\end{cor}
\begin{proof} Given $\pi:E\to X$ with the extension property, just consider the map $\phi:=T\pi :\cantor\to X$.
\end{proof}

\smallskip
At this point, we digress a little bit by showing that a result of the same kind holds true for \emph{Polish} dynamical systems, the ``universal" space being (as should be expected) the Baire space $\NN^\NN$. This may be quite well known, but we couldn't locate a reference.
\begin{prop} For any Polish space $X$, there exists a map $\pi:\NN^\NN\to X$ with the extension property. Consequently, any self-map $T:X\to X$ is a factor of a self-map of $\NN^\NN$.
\end{prop}
\begin{proof} We show that the ``usual" construction of a continuous surjection from $\NN^\NN$ onto $X$ produces a map with the extension property.

Denoting by $\NN^{<\NN}$ the set of all finite sequences of integers, let $(V_t)_{t\in\NN^{<\NN}}$ be a family of non-empty open subsets of $X$ satisfying the following properties:

\begin{itemize}
\item $V_\emptyset=X$ and $\diam(V_t)<2^{-\vert s\vert}$ for all $t\neq\emptyset$;
\item $\overline{V}_{ti}\subseteq V_s$ for each $t$ and all $i\in\NN$;
\item $\bigcup_{i\in\NN} V_{ti}=V_t$ for all $t$.
\end{itemize}

Let $\pi:\NN^\NN\to X$ be the map defined by $\{\pi(\alpha)\}:=\bigcap_{k\geq 1} \overline V_{\alpha|_k}$. We show that this map $\pi$ has the extension property.

Let us fix a map $\phi:\NN^\NN\to X$. Since $\phi$ is continuous, one can find, for each $k\in\NN$, a family $\mathcal S_k\subseteq\NN^{<\NN}$ such that $\bigcup\{ [s];\; {s\in\mathcal S_k}\}=\NN^\NN$ and each set $\phi([s])$, $s\in\mathcal S_k$ is contained in some open set $V_{t(s)}$ with $\vert t(s)\vert=k$. Moreover, we may do this in such a way that the following properties hold true:
\begin{itemize}
\item if $k\geq 2$, every $s\in\mathcal S_k$ is an extension of some $\widetilde s\in \mathcal S_{k-1}$;
\item if $s\in\mathcal S_k$ is an extension of $\widetilde s\in\mathcal S_{\widetilde k}$,
 (with $\widetilde k\leq k$), then $t(s)$ is an extension of $t(\widetilde s)$;
 \item $\mathcal S_k$ is an antichain, \mbox{\it i.e.} no $s\in\mathcal S_k$ is a strict extension of an $s'\in\mathcal S_k$.
\end{itemize}

(To get the third property, just remove from $\mathcal S_k$ the unnecessary $s$, \mbox{\it i.e.} those extending some $s'\in\mathcal S_k$ with $s'\neq s$.)

Note that for each fixed $k\in\NN$, the family $\{ [s];\; s\in\mathcal S_k\}$ is actually a partition of $\NN^\NN$ because $\mathcal S_k$ is an antichain; hence, for any $\alpha\in\NN^\NN$, we may denote by $s_k(\alpha)$ the unique $s\in\mathcal S_k$ such that $s\subseteq \alpha$. Now, define $S:\NN^\NN\to\NN^\NN$ as follows: if $\alpha\in\NN^\NN$ then $S(\alpha)|_k:= t(s_k(\alpha))$ for all $k\in\NN$. It is readily checked that $S$ is continuous and $\phi=\pi S$.
\end{proof}

\smallskip
Let us now come back to compact metrizable dynamical systems. 
From Corollary \ref{BalcSim}, one can easily prove that there is a ``universal extension" for all dynamical systems $(X,T)\in\mathfrak D$. However, this map is defined on a huge compact \emph{nonmetrizable} space. In what follows, for any compact metric spaces $X$ and $Y$, we endow $\mathcal C(X,Y)$ with the topology of uniform convergence, which turns it into a Polish space.

\begin{prop}\label{cheating} There exists a compact $0$-dimensional space $\mathbf E$ and a map $U:\mathbf E\to \mathbf E$ such that every $(X,T)\in\mathfrak D$ is a factor of $(\mathbf E,U)$.
\end{prop}
\begin{proof} By Corollary \ref{BalcSim}, it is enough to show that there exists a $0$-dimensional compact dynamical system $(\mathbf E,U)$ which is an extension of every \emph{Cantor} dynamical system $(\cantor ,S)$.

\smallskip
Consider the space $\mathbf E:= (\cantor)^{\cont{\cantor}{\cantor}}$ endowed with the product topology. This is a $0$-dimensional compact space. For any 
$z=(z_S)_{S\in\mathcal C(\cantor,\cantor)} \in \mathbf E$, we define $U(z) \in
	\mathbf E$ as follows:
	$$U(z)_S=S(z_S)\qquad\hbox{for every $S\in\mathcal C(\cantor,\cantor)$}\, .$$ 
	
	Then, $U$ is continuous because each $z\mapsto U(z)_S$ is continuous. Moreover, each $S\in \cont{\cantor}{\cantor}$ is indeed a factor of  $U$ by the very definition of $\mathbf E$ and $U$: just consider the projection map $\pi_S:\mathbf E\to\cantor$ from $\mathbf E$ onto the $S$-coordinate.
\end{proof}

\begin{remark} If one just wants to lift those dynamical systems $(X,T)\in\mathfrak D$ which are \emph{transitive}, \mbox{\it i.e.} have dense orbits, then one can take $(\mathbf E,U):=(\beta\mathbb Z_+, \sigma)$, where $\beta\mathbb Z_+$ is the Stone-$\check{\rm C}$ech compactification of $\mathbb Z_+$ and $\sigma :\beta\mathbb Z_+\to\beta\mathbb Z_+$ is the unique continuous extension of the map $n\mapsto 1+n$. Indeed, given a point $x_0\in X$ with dense $T$-orbit, the map $n\mapsto T^n(x_0)$ can be extended to a continuous map $\pi_T :\beta\mathbb Z_+\to X$ showing that $(X,T)$ is a factor of $(\beta\mathbb Z_+,\sigma)$.
\end{remark}

\begin{remark}\label{biensur} If $(\mathbf E,U)$ is a compact dynamical system which is universal for $\mathfrak D$ in the sense of Proposition \ref{cheating}, then the space $\mathbf E$ cannot be metrizable.
\end{remark}
\begin{proof} For any $g\in\TT$, let us denote by $R_g:\TT\to\TT$ the associated rotation of $\TT$. We show that if $(\mathbf E, U)$ is a compact dynamical system that is ``only" a common extension of all rotations $R_g$, then $\mathbf E$ already cannot be metrizable. 

Let us fix a point $e\in\mathbf E$.
 Since $\bf E$ is assumed to be compact and metrizable, one can find an increasing sequence of integers $(i_k)$ such that the sequence $U^{i_k}(e)$ is convergent. Now, given $g\in\TT$, let $\pi_g:\mathbf E\to \TT$ be a map witnessing that $U$ is an extension of $R_g$. Setting $\xi_g:=\pi_g(e)$, we then have $g^{i_k}\xi_g=\pi_gU^{i_k}(e)$ for all $k$; and since $\pi_g$ is continuous, it follows that for every $g\in\TT$, the sequence $g^{i_k}$ is convergent. But it is well known that this is not possible. (One can prove this as follows: if $g^{i_k}\to \phi(g)$ pointwise, then $\int_\TT \phi(g)g^{-i_k}dg\to 1$ by dominated convergence, contradicting the Riemann-Lebesgue lemma.)
\end{proof}

\subsection{Lifting to the Baire space}\label{Bairelift} As observed above, if one wants to lift too many maps at the same time then one cannot do it with a compact metrizable space. On the other hand, our next result says that for ``small families" of maps, one can find a common extension defined on a \emph{Polish} space.

\smallskip
In what follows, for any compact metric spaces $X$ and $Y$, we endow $\mathcal C(X,Y)$ with the topology of uniform convergence, which turns it into a Polish space.

\begin{thm}\label{Ksigma1} Let $X$ be a compact metric space. Given any $\sigma$-compact family $\mathcal F\subseteq \mathcal C(X,X)$, one can find a map $U_{\mathcal F}:\baire\to\baire$ such that every $T\in\mathcal F$ is a factor of $U_{\mathcal F}$.
\end{thm}

This  will follow from the next three lemmas, the first of which is a refinement of Theorem \ref{Weihr}.

\smallskip
Let us say that a map $\pi:E\to X$ between compact metric spaces has the {\em strong extension property} if for every given map $\Phi :\cantor \rightarrow \cont{E}{X}$,  there exists a map $F:\cantor\rightarrow \cont{X}{X}$ such that for all $p \in \cantor$ we have that $\Phi(p)=\pi F(p)$.

\begin{lem}\label{lem:strongext} Let $X$ be a compact metric space. Then, there is a map $\pi: \cantor \rightarrow X$ which has the strong extension property.
\end{lem}
\begin{proof} We follow the proof of Theorem \ref{Weihr} given in \cite[p. 110]{kurka}. The construction of the map $\pi$ is the same as in that proof. We only need to verify that $\pi$ has the additional stronger property.
	
	We recall the construction of $\pi$. As in \cite{kurka}, one can easily construct inductively a sequence $\{J_i\}$ of finite sets of cardinality at least $2$
	and collections $\{V_{s}\}$ and $\{W_{s}\}$ of nonempty open subsets of $X$ such that the following conditions hold for all $k \ge 1$ and for all $s \in
	\prod _{i=1}^k J_i$:
	\begin{itemize}
		\item the diameters of $V_{s}$ and $W_{s}$ are less than $2^{-k}$;
		\item $\left\{V_{t};\; t \in \prod _{i=1}^k J_i\right\}$ is a cover of $X$;
		\item $\left\{W_{sj};\; j\in J_{k+1} \right\}$ is a cover of ${\overline V}_{s}$;
		\item $V_{sn}= V_{s} \cap W_{sn}$.
	\end{itemize}

 We let $\Lambda: =\prod_{i=1}^{\infty} J_i$, which is homeomorphic to $\cantor$. For $\alpha \in \Lambda$, we define $\pi(\alpha)$ in the obvious way: $\{ \pi(\alpha)\}= \cap _{k=1}^{\infty} {\overline V} _{\alpha |_k}$. Then, $\pi$ is a well-defined map from $\Lambda$ onto $X$. 
	
	Now let $\Phi: \cantor \rightarrow \cont{\Lambda}{X}$ be a map. With each $p \in \cantor$ and $\Phi(p)$ we will associate a map $F(p): \Lambda \rightarrow \Lambda$ as in \cite{kurka}. However, we must do this in a continuous fashion. 
	
		First, for each $k \ge 1$, we let $2^{-k}> \varepsilon_k >0$ be sufficiently small so that $\varepsilon_k$ is a Lebesgue number for the covering  $\{W_{s j};\; j\in J_{k+1}\}$ of ${\overline V}_{s}$ for all $s \in
	\prod _{i=1}^k J_i$. 
	
	For any fixed $p \in \cantor$,  the uniform continuity of $\Phi(p)$ allows us to choose an increasing sequence  of positive integers $\{m_k\}$ such that for each $k \ge 1$ and $ s \in \prod_{i=1}^{m_k} J_i$, the diameter of the set $\Phi(p)([s])$ is less than $\varepsilon_k/4$. (Here $[s]$ denotes those elements of $\Lambda$ which begin with $s$.)
	Moreover, as the map $\Phi$ is continuous, $\Phi(\cantor)$ is a uniformly equicontinuous subset of $\cont{\Lambda}{X}$; so we may in fact choose $\{m_k\}$ independent of $p$. Hence, we have an increasing sequence of positive integers $\{m_k\}$ such that \emph{for all $p \in \cantor$} and all $k \ge 1$ and $ s \in \prod_{i=1}^{m_k} J_i$, the diameter of the set $\Phi(p)([s])$ is less than $\varepsilon_k/4$. 
	
	By the uniform continuity of $\Phi$ we can choose an increasing sequence of integers $\{l_k\}$ such that if $p, p' \in \cantor$ satisfy $p|_{l_k} = p'|_{l_k}$, then $d(\Phi(p),\Phi(p'))< \varepsilon_k/4$. (Here, $d$ is the ``sup" metric on $\cont{\Lambda}{X}$.) Note that this implies in particular that for all $q\in\{ 0,1\}^{l_k}$ and $ s \in \prod_{i=1}^{m_k} J_i$, the diameter of the set 
	$E_{s,q}:=\bigl\{ \Phi(p)([s]);\; p|_{l_k}=q\bigr\} $ is less than $\varepsilon_k$. This, in turn, implies (by the choice of $\varepsilon_k$) that $E_{s,q}$ is contained in an open set $V_t$ for some $t  \in \prod_{i=1}^{k} J_i$; and, moreover, that if $E_{s,q}$ was already known to be contained in $V_{\widetilde t}$ for some $\widetilde t\in \prod_{i=1}^{k-1} J_i$, then $t$ may be chosen to be an extension of $\widetilde t$. Let us summarize what we have: 
	
	\begin{quote}	(i) 
	For all $ k \ge 1$,  for all $q \in \{0,1\}^{l_k}$ and  for all $s \in \prod_{i=1}^{m_k} J_i$, there exists 
	$t = t(q, s) \in \prod_{i=1}^{k} J_i$ such that for all $p \in \cantor$ extending $q$, we have that $\Phi(p) ([s]) \subset V_{t}$.
	\end{quote}
	
	\begin{quote}\label{quote:ext} (ii) Everything can be done recursively in such a way that
	if  $k' \ge k$, $q'\in \{0,1\}^{l_{k'}}$ is an extension of $q$ and $s' \in \prod_{i=1}^{m_{k'}} J_i$ is an extension of $s$, then $t' = t'(q', s')$ is an extension of $t$.
	\end{quote}
	
	Let us now define the desired $F:\cantor \rightarrow \cont{\Lambda}{\Lambda}$. Fix $p \in \cantor$. Then, $F(p): \Lambda \rightarrow \Lambda$ is defined as follows. For each $\alpha \in \Lambda$, we define $F(p)(\alpha)$ so that for all $k \ge 1$, $F(p)(\alpha)|_k = t(p|_{l_k}, \alpha|_{m_k})$. By the above, $F(p)$ is a well-defined function, and it is continuous. Moreover, $F$ itself is continuous: if $p,p' \in \cantor$ with $p|_{l_k} = p'|_{l_k}$, then $d(F(p),F(p'))< \varepsilon_k < 2^{-k}$. That $\Phi(p) = \pi F(p)$ for all $p\in\cantor$ follows in the same manner as in \cite{kurka}.
	\end{proof}
\begin{cor}\label{cor:intermediate}
	Let $X$ be a compact metric space and let $\mathcal M$ be a compact subset of $ \cont{X}{X}$. Then, there is a compact set $\mathcal N \subseteq \cont{\cantor}{\cantor}$ such that every $T\in\mathcal M$ is a factor of some $S\in\mathcal N$. 
\end{cor}
\begin{proof} Since $\mathcal M$ is compact and metrizable, there is a map $G:\cantor \to\cont{X}{X}$ such that $\mathcal M=G(\cantor)$. Now, choose a map $\pi: \cantor \rightarrow X$ having the strong extension property, and consider the map $\Phi:\cantor \rightarrow \cont{\cantor}{X}$ defined by $\Phi(p)=G(p) \pi$. Let $F:\cantor \rightarrow \cont{\cantor}{\cantor}$ be  a map such that $\Phi(p) =\pi F(p)$ for all $p \in \cantor$. Then $G(p)\pi= \Phi(p) = \pi F(p)$, so that $G(p)$ is a factor of $F(p)$ for every $p\in \cantor$ (because $\pi$ is onto). Hence, we may take $\mathcal N:=F(\cantor)$. 
	\end{proof}
	
	\begin{lem}\label{lem:cantorfactor} Let ${\mathcal N} $ be a compact subset of $ \cont{\cantor}{\cantor}$. Then, there exists map $U_{\mathcal N}: \baire \rightarrow \baire$ such that every $S\in\mathcal N$ is factor of $U_{\mathcal N}$.
	\end{lem}
	\begin{proof} Recall first that the Polish space $\cont{\cantor}{\cantor}$ is homeomorphic to $\baire$. Indeed, it is easy to check that $\cont{\cantor}{\cantor}$ is $0$-dimensional (because $\cantor$ is), and it is also easy to convince oneself that every equicontinuous family of self-maps of $\cantor$ has empty interior in $\cont{\cantor}{\cantor}$, so that $\cont{\cantor}{\cantor}$ is nowhere locally compact.
	
	As ${\mathcal N}$ is a compact subset of $\cont{\cantor}{\cantor}\simeq\baire$, we may enlarge it if necessary  and assume that ${\mathcal N}$ is homeomorphic to $\cantor$. Then $\mathbf Z:= \cont{{\mathcal N}}{\cantor}$ is homeomorphic to $\baire$. Now we define $U_{\mathcal N}: \mathbf Z \rightarrow \mathbf Z$ as in the proof of Proposition \ref{cheating} (changing the notation slightly):  if $z \in \mathbf Z$ then $U_{\mathcal N}(z)(S):=S(z(S))$ for every $S \in {\mathcal N}$. Then $U_{\mathcal N}$ is easily seen to be continuous (recall that $\mathbf Z=\cont{\mathcal N}{\cantor}$ is endowed with the topology of uniform convergence), and the evaluation map $\pi_S$ at $S$ shows that each $S\in\mathcal N$ is a factor of $U_{\mathcal N}$. (To see that $\pi_S:\mathbf Z\to \cantor$ is onto, consider constant maps $z\in\mathbf Z$.)
	\end{proof}

\smallskip
\begin{lem}\label{lem:okmap} Let $\mathbf Z :=\cantor$ or $\baire$, and let $U_1, U_2, \ldots$ be maps from $
\mathbf Z$ to $\mathbf Z$. Then, there is a map $U: \mathbf Z  \rightarrow \mathbf Z$ such that each $U_n$ is a factor of $U$.
\end{lem}
\begin{proof} The only property needed on the space $\mathbf Z$ is that $\mathbf Z$ is homeomorphic to $\mathbf Z^\NN$.
		 Define $U: \mathbf Z^\NN \rightarrow \mathbf Z^\NN$ by $U(z_1,z_2, \ldots): = (U_1(z_1),U_2(z_2), \ldots)$. Then, $T$ is continuous. Moreover, if $\pi_n:\mathbf Z^\NN \rightarrow \mathbf Z$ is the projection of  $\mathbf Z^\NN$ onto the $n$-th coordinate, we have that 
	$\pi_nU= U_n\pi_n$ for all $n\geq 1$. Identifying $\mathbf Z^\NN$ with $\mathbf Z$, this shows that $U$ has the required property.	
\end{proof}

\begin{proof}[Proof of Theorem \ref{Ksigma1}] Let $\mathcal F$ be a $\sigma$-compact of $\cont{X}{X}$, and write ${\mathcal F} = \bigcup_n{\mathcal M}_n$ where ${\mathcal M}_n$ are compact sets. By Corollary~\ref{cor:intermediate}, we may choose compact sets ${\mathcal N}_n \subseteq \cont{\cantor}{\cantor}$ such that every $T\in{\mathcal M}_n$ is a factor of some $S\in{\mathcal N}_n$. By Lemma~\ref{lem:cantorfactor}, we may choose for each $n$ a map $U_n: \baire \rightarrow \baire$ such that every $S\in {\mathcal N}_n$ is a factor of $U_n$. Now, by Lemma~\ref{lem:okmap}, there is a map $U: \baire \rightarrow \baire$ such that each map $U_n$ is factor of $U$. Since factors of factors are again factors, we may thus take $U_{\mathcal F}:=U$.
\end{proof}

\subsection{Lifting to the Cantor space}\label{nonsense}
{{The appearance of the Baire space $\NN^\NN$ in Theorem \ref{Ksigma1} may look rather surprising since we are dealing with self-maps of \emph{compact} metric spaces. This suggests the following question.
\begin{question}\label{faitchier} Let $X$ be a compact metric space. For which families $\mathcal F\subseteq\cont{X}{X}$ is it possible to find a common extension $U$ defined on the Cantor space $\cantor$?
\end{question}
}}

Let us denote by $\hbox{\textgoth I}$ the family of all subsets $\mathcal F$ of $\cont{X}{X}$ having the above property. It follows from Lemma \ref{lem:okmap} that $\hbox{\textgoth I}$ is closed under countable unions; and clearly, $\hbox{\textgoth I}$ is hereditary for inclusion. Hence, $\hbox{\textgoth I}$ is a \emph{$\sigma$-ideal} of subsets of $\cont{X}{X}$. Since $\sigma$-ideals of sets are quite classical objects of study, it might be interesting to investigate the structural properties of this particular example.

\smallskip
Having said that, we are quite far from being able to answer Question \ref{faitchier} in a satisfactory way. However, in this section we give a rather special sufficient condition for belonging to $\hbox{\textgoth I}$.

\smallskip
 In what follows, we fix a compact metric space $(X,d)$, and we also denote by $d$ the associated ``sup" metric on $\cont{X}{X}$.

\smallskip
For each $i\in\mathbb Z_+$, let us denote by $\mathtt e_i:\cont{X}{X}\to \cont{X}{X}$ the map defined by
$$\mathtt e_i (S):= S^{i}=\underbrace{S\circ\cdots \circ S}_{i\, times}\, .$$
We shall say that a family $\mathcal N\subseteq \cont{X}{X}$ has \emph{uniformly controlled powers} if the sequence $\{ \mathtt e_i|_\mathcal N;\; i\geq 0\}\subseteq\cont{\mathcal N}{\cont{X}{X}}$ is pointwise equicontinuous; in other words, if whenever a sequence $(S_k)\subseteq\mathcal N$ converges to some $S\in\mathcal N$, it holds that
\begin{equation}\label{truc}
S^i_k(x)\xrightarrow{k\to\infty} S^i(x)\qquad\hbox{uniformly with respect to $x$ and $i$}. 
\end{equation}

\begin{ex} Obviously, every finite family $\mathcal N\subseteq\cont{X}{X}$ has uniformly controlled powers.
\end{ex}

\begin{ex}\label{exjoke} Let $\mathcal N$ be a compact subset of $\cont{X}{X}$, and assume the each map $S\in\mathcal N$ has the following properties: $S$ is $1$-Lipschitz with a unique fixed point $\alpha(S)$ in $X$, and $S^i(x)\to \alpha(S)$ for every $x\in X$. (This holds for example if $d(S(x),S(y))<d(x,y)$ whenever $x\neq y$.) Then $\mathcal N$ has uniformly controlled powers.
\end{ex}
\begin{proof} Observe first that  the map $S\mapsto \alpha(S)$ is continuous on $\mathcal N$ (because it has a closed graph and acts between compact spaces). Moreover, $S^i(x)\to\alpha(S)$ \emph{uniformly} with respect to $S$ and $x$ as $i\to\infty$. (This follows from Dini's theorem applied to the continuous functions $\Phi_i:\mathcal N\times X\to\mathbb R$ defined by $\Phi_i(S,x):=d(S^i(x),\alpha(S))$. The sequence $(\Phi_i)$ is nonincreasing because each $S\in \mathcal N$ is $1$-Lipschitz.) So there is a sequence $(\varepsilon_i)$ tending to $0$ such that 
$$d(S^i(x),\alpha(S))\leq\varepsilon_i\qquad\hbox{for all $x$, $i$ and $S \in \mathcal N$}.$$

Now, let $(S_k)$ be a sequence in $\mathcal N$ converging (uniformly) to some map $S$. Set $\alpha_k:=\alpha(S_k)$ and $\alpha:=\alpha(S)$. For every $x$ and $i$, we have on the one hand
$$d(S_k^i(x),S^i(x))\leq d( S_k^i,S^i),$$
 and on the other hand
\begin{eqnarray*}
d(S_k^i(x),S^i(x))&\leq&d(S_k^i(x),\alpha_k)+d(\alpha_k,\alpha)+d(\alpha,S^i(x))\\
&\leq&2\varepsilon_i +d(\alpha_k,\alpha)\, .
\end{eqnarray*}

Since $\alpha_k\to\alpha$  and $S_k^i\to S^i$ for each fixed $i\in\mathbb Z_+$, this gives (\ref{truc}).
\end{proof}

\begin{prop}\label{propjoke} Let $\mathcal F\subseteq\cont{X}{X}$. Assume that $\mathcal F$ can be covered by countably many compact sets $\mathcal N$ having uniformly controlled powers. Then $\mathcal F$ admits a common extension $U_{\mathcal F}$ defined on $\cantor$.
\end{prop}
\begin{proof} By Lemma \ref{lem:okmap}, it is enough to prove the result for each one of the compact families $\mathcal N$. Moreover, since every self-map of a compact metric space is a factor of some self-map of $\cantor$, it is enough to show that $\mathcal N$ admits a common extension $U$ defined on some compact metric space $\mathbf E$.

For any $i\in\mathbb Z_+$ and $a\in X$, denote by $\mathtt e_{i,a}:\mathcal N\to X$ the map defined by $\mathtt e_{i,a}(S):= S^i(a)$. Since $\mathcal N$ has uniformly controlled powers, the family of all such maps $\mathtt e_{i,a}$ is pointwise equicontinuous on $\mathcal N$; so its uniform closure $\mathbf E$ is a compact subset of $\mathbf Z=C(\mathcal N,X)$, by Ascoli's theorem. Note also that $\mathbf E$ contains all constant maps because $\{ \mathtt e_{i,a}\}$ already does (take $i=0$ and let $a$ vary).

Now, consider the map $U_{\mathcal N}:\mathbf Z\to \mathbf Z$ defined in the proof of Lemma \ref{lem:cantorfactor}, 
$$U_{\mathcal N}(z)(S)=S(z(S))\, .$$

The family $\{\mathtt e_{i,a}\}$ is invariant under $U_{\mathcal N}$, because $U_{\mathcal N}(\mathtt e_{i,a})=\mathtt e_{i+1,a}$. Hence, $\mathbf E$ is also invariant under $U_{\mathcal N}$. Moreover, since 
$\mathbf E$ contains all constant maps, each evaluation map $\pi_S:\mathbf E\to X$ is onto. So the map $U:=U_{\mathcal N}|_{\mathbf E}$ has the required property.
\end{proof}

\begin{cor}\label{cor:contract} The family $\mathcal F$ of all $d$-contractive self-maps of $X$ has a common extension defined on $\cantor$.
\end{cor}
\begin{proof} For each $c\in (0,1)$, let us denote by $\mathcal N_c$ the family of all $c$-Lipschitz self-maps of $X$. Then $\mathcal N_c$ is compact by Ascoli's theorem, and it has uniformly controlled powers by Example \ref{exjoke}. Since $\mathcal F=\bigcup_{n\in\NN} \mathcal N_{1-\frac1n}$, the resut follows.
\end{proof}

\begin{remark} In general, the family of all $1$-Lipschitz self-maps of $X$, or even the family of all isometries of $X$, does not have a common extension to $\cantor$. Consider for example the rotations of $X:=\TT$ (see the proof of Remark \ref{biensur}).
\end{remark}

We conclude this section with a kind of abstract nonsense characterizing the belonging to a subclass of the $\sigma$-ideal $\hbox{\textgoth I}$. 

\begin{fact} Let $\mathcal N$ be a compact subset of $\cont{X}{X}$, and let  $U_{\mathcal N}:\cont{\mathcal N}{X}\to \cont{\mathcal N}{X}$ be the (usual) map defined by $U_{\mathcal N}(z)(S)=S(z(S))$. The following are equivalent.
\begin{enumerate}
\item[\rm (i)] There is a map $\mathfrak U:\cantor\to\cantor$ and a continuous map $S\mapsto \mathfrak q_S$ from $\mathcal N$ into $\cont{\cantor}{X}$ such that $\mathfrak U$ is an extension of each $S\in\mathcal N$, with factoring map $\mathfrak q_S$.
\item[\rm (ii)] There is a compact set $\mathbf E\subseteq\cont{\mathcal N}{X}$ invariant under $U_{\mathcal N}$ such that, for each $S\in\mathcal N$, the evalution map $z\mapsto z(S)$ maps $\mathbf E$ onto $X$.
\end{enumerate}
\end{fact}
\begin{proof} Since every compact metrizable dynamical system has an extension to $\cantor$, the implication (ii)$\implies$(i) follows easily from the proof of Proposition \ref{propjoke}.

\smallskip
Conversely, assume that (i) holds true, and denote by $\mathfrak q : \mathcal N\to\cont{\cantor}{X}$ the map $S\mapsto \mathfrak q_S$. For any $p\in\cantor$, denote by $\delta_p:\cont{\cantor}{X}\to X$ the evaluation map at $p$. Then, the map $p\mapsto \delta_p\circ \mathfrak q$ is continuous from $\cantor$ into $\cont{\mathcal N}{X}$. So the set $\mathbf E:=\{ \delta_p\circ \mathfrak q;\; p\in\cantor\}$ is a compact subset of $\cont{\mathcal N}{X}$. Moreover, for any $z=\delta_p\circ \mathfrak q\in\mathbf E$ and all $S\in\mathcal N$ we have
$$U_{\mathcal  N}(z)\,(S)=S(z(S))=S(\mathfrak q_S(p))=\mathfrak q_S(\mathfrak U(p))=\delta_{\mathfrak U(p)}\circ\mathfrak q\, (S);$$
and this show that $\mathbf E$ is invariant under the map $U_{\mathcal N}$. Finally, the evaluation map at any $S\in\mathcal N$ maps $\mathbf E$ onto $X$ because $\mathfrak q_S$ is onto.
\end{proof}

\section{Common generalized extensions}\label{generalized} In this section, we show that something similar to Theorem \ref{Ksigma1} does hold true with $\cantor$ in place of $\NN^\NN$, but with a weaker notion of ``factor".

\smallskip
For any compact metric space $X$, let $\hyp{X}$ be the space of all nonempty compact subsets of $X$ endowed with the Haudorff metric. If $T:X\to X$ is a map, we denote again by $T:\hyp{X}\to\hyp{X}$ the map induced by $T$ on $\hyp{X}$; that is, $T(M)=\{ T(x);\; x\in M\}$ for every $M\in\hyp{X}$.  Obviously, $(X,T)$ is equivalent to a subsystem of $(\hyp {X}, T)$ thanks to the embedding $X\ni x\mapsto \{ x\}\in\hyp{X}$. However, in general $(X,T)$ is not a \emph{factor} of $(\hyp {X}, {T})$. For example, this cannot be the case if $T$ is onto (so that $X$ is a  fixed point of ${T}$ in $\hyp{X}$) and $T$ has no fixed point in $X$. 

\smallskip
We shall say that $(X,T)$ is a a \emph{generalized factor} of a dynamical system $(E,S)$ if  there exists an upper semi-continuous mapping $\pi:E \rightarrow \hyp{X}$ such that
the following conditions hold:
\begin{itemize}
	\item $\{x\}\in \pi(E)$ for every $x\in X$;
	\item $\pi(S(u))={T}(\pi(u))$ for all $u\in E$.
\end{itemize}

(Recall that $\pi:E\to\hyp{X}$ is \emph{upper semicontinuous} if the relation $x\in \pi(u)$ is closed in $E\times X$; equivalently, if for any open set $O\subseteq X$, the set $\{ u\in E\;:\; \pi(u)\subseteq O\}$ is open in $E$.)

\smallskip
It is obvious from the definition that ``true" factors are generalized factors. Also, we have the expected transitivity property:
\begin{remark} If $(X,T)$ is a generalized factor of $(Y,S)$ and $(Y,S)$ is a generalized factor of $(Z,R)$, then $(X,T)$ is a generalized factor of $(Z,R)$.
\end{remark}
\begin{proof} Let $\pi: Y \rightarrow \hyp{X}$ witness that $T$ is a generalized factor of $S$ and let $\pi':Z \rightarrow \hyp{Y}$ witness the fact that $S$ is a generalized factor of $R$. Define $\pi^* : \hyp{Y} \rightarrow  \hyp{X}$ by
$\pi^*(M) := \bigcup \{\pi(y): y \in M\}$. Then, $\pi^*$ is a well-defined upper semicontinuous mapping. Now, the mapping $\gamma: Z \rightarrow \hyp{X}$ defined by $\gamma := \pi^*\circ \pi'$ is is easily checked to be again upper semicontinuous, and shows that $T$ is a generalized factor of $R$. 
\end{proof}

\begin{thm}\label{Ksigma2} Let $X$ be a compact metric space. Given any $\sigma$-compact family $\mathcal F\subseteq \mathcal C(X,X)$, one can find a map $U_{\mathcal F}:\cantor\to\cantor$ such that every $T\in\mathcal F$ is a generalized factor of $U_{\mathcal F}$.
\end{thm}

The proof will use Corollary \ref{cor:intermediate}, Lemma \ref{lem:okmap} and the following variant of Lemma \ref{lem:cantorfactor}.

\begin{lem}\label{cantorgenfactor} Let ${\mathcal N} $ be a compact subset of $ \cont{\cantor}{\cantor}$. Then, there exists map $U_{\mathcal N}: \cantor \rightarrow \cantor$ such that every $S\in\mathcal N$ is generalized factor of $U_{\mathcal N}$.
\end{lem}
\begin{proof} We assume of course that $\mathcal N\neq\emptyset$. Let us denote by $\mathbf Z$ be the set of all compact subsets of ${\mathcal N} \times \cantor$ whose projection on the first coordinate is all of ${\mathcal N}$. This is a closed subset of $\hyp{\mathcal N\times\cantor}$. Moreover, $\mathbf Z$ is $0$-dimensional because ${\mathcal N} \times {\cantor}$ is, and  it is easily seen to have no isolated points. Hence, $\mathbf Z$ is homeomorphic to $\cantor$. 

We define $U_{\mathcal N}: \mathbf Z \rightarrow \mathbf Z$ in the following manner. Let $A \in \mathbf Z$ and for each $ S\in {\mathcal N}$, let $A_S$ be the vertical cross section of $A$ at $S$, i.e.,
	 $A_S =\{ x\in\cantor\;:\; (S,x) \in A\}$. We define the set $U_{\mathcal N}(A)\subseteq \mathcal N\times\cantor$ by its vertical cross sections: $U_{\mathcal N}(A)_S :=S(A_S)$ for every $S\in\mathcal N$. In other words,
	 $$ U_{\mathcal N}(A)=\bigcup_{S\in\mathcal N} \{ S\}\times S(A_S)=\{ (S, S(x));\; (S,x)\in A\}\, .$$
	 
	 Let us check that $U_{\mathcal N}(A)\in\mathbf Z$. That the projection of $\mathcal U_{\mathcal N}(A)$ on the first coordinate is all of $\mathcal N$ is obvious because $A_S\neq\emptyset$ for every $S\in\mathcal N$. Let $\bigl\{ (S_n,u_n)\bigr\}$ be a sequence in $U_{\mathcal N}(A)$ converging to some point $(S,u)\in\mathcal N\times\cantor$. Then $u_n=S_n(x_n)$ for some $x_n\in\cantor$ such that 
	 $(S_n,x_n)\in A$. By compactness, we may assume that $x_n$ converges to some $x\in\cantor$. Then $(S,x)\in A$ because $A$ is closed in $\mathcal N\times\cantor$; and $u=S(x)$ because 
	 $u_n\to u$ and $S_n\to S$ uniformly. So $u\in S(A_S)$, \mbox{\it i.e.} $(S,u)\in U_{\mathcal N}(A)$. Thus, we see that $U_{\mathcal N}(A)$ is a closed and hence compact subset of $\mathcal N\times\cantor$.
	 
	 The same kind of reasoning shows that $U_{\mathcal N}:\mathbf Z\to\mathbf Z$ is continuous. One has to check that if $A_n\to A$ in $\hyp{\mathcal N\times\cantor}$ then: (i) for every $(S,u)\in U_{\mathcal N}(A)$, one can find a sequence $\bigl\{ (S_n,u_n)\bigr\}$ with $(S_n,u_n)\in U_{\mathcal N}(A_n)$ such that $(S_n,u_n)\to (S,u)$; and (ii) for every sequence $\bigl\{ (S_{n_k},u_{n_k})\bigr\}$ with $(S_{n_k},u_{n_k})\in U_{\mathcal N}(A_{n_k})$ converging to some $(S,u)\in \mathcal N\times\cantor$, we have $(S,u)\in U_{\mathcal N}(A)$. Both points follow easily from the definition of $U_{\mathcal N}$ and the fact that $A_n\to A$.

	Now, fix $S\in\mathcal N$ and consider $\pi_S:\mathbf Z\to \hyp{\cantor}$ defined by $\pi_S(A):=A_S$. This is an upper semicontinuous mapping; and $\pi_S U_{\mathcal N}= S\pi_S$ by the very definition of $U_{\mathcal N}$. Moreover, $\{ x\}=\pi_S(\mathcal N\times\{ x\})$ for every $x\in\cantor$. This shows that $S$ is a generalized factor of $U_{\mathcal N}$.
\end{proof}

\begin{proof}[Proof of Theorem \ref{Ksigma2}] One can now proceed exactly as in the proof of Theorem \ref{Ksigma1} (using Lemma \ref{cantorgenfactor} instead of Lemma \ref{lem:cantorfactor}), keeping in mind that factors are generalized factors and that generalized factors of generalized factors are again generalized factors.
\end{proof}

\end{document}